\documentclass[12pt,reqno]{amsart}

\usepackage{amsfonts,appendix, latexsym, amssymb, amsmath, fullpage, bbm}
\usepackage{hyperref}

\numberwithin{equation}{section}

\newcommand{\R}{\mathbb R}

\newcommand{\E}{\mathbb E}
\def \P {\mathbb{P}}

\def \e {\varepsilon}

\newcommand{\tr}{\textnormal{tr}}

\DeclareMathOperator*{\adj}{adj}

\newtheorem{theorem}{Theorem}[section]

\newtheorem{lemma}{Lemma}[section]

\theoremstyle{remark}
\newtheorem{remark}[theorem]{Remark}

\begin{document}


\title{Smoothed analysis of symmetric random matrices with continuous distributions}

\author{Brendan Farrell}
\address{Computing and Mathematical Sciences, California Institute of Technology, 1200 E. California Blvd., Pasadena, CA 91125, U.S.A. }
\email{farrell@cms.caltech.edu}
\author{Roman Vershynin}
\address{Department of Mathematics, University of Michigan, 530 Church St., Ann Arbor, MI 48109, U.S.A.}
\email{romanv@umich.edu}

\subjclass[2010]{60B20,15B52}

\date{\today}

\begin{abstract}
  We study invertibility of matrices of the form $D+R$ where $D$ is an arbitrary symmetric deterministic 
  matrix, and $R$ is a symmetric random matrix whose independent entries have continuous distributions with bounded densities. 
  We show that $\|(D+R)^{-1}\| = O(n^2)$ with high probability. The bound is completely independent of $D$. 
  No moment assumptions are placed on $R$; in particular the entries of $R$ can be arbitrarily heavy-tailed. 
\end{abstract}

\maketitle

\section{Introduction}

This note concerns the invertibility properties of $n \times n$ random matrices of the type 
$D+R$, where $D$ is an arbitrary deterministic matrix and $R$ is a random matrix
with independent entries.  
What is the typical value of the spectral norm of the inverse, $\|(D+R)^{-1}\|$?

This question is usually asked in the context of {\em smoothed analysis} of algorithms \cite{ST ICM}.
There $D$ is regarded as a given matrix, possibly poorly invertible, and $R$ models random noise.
Heuristically, adding noise should improve invertibility properties of $D$, so 
the typical value $\|(D+R)^{-1}\|$ should be nicely bounded for any $D$. 
Sometimes this is true, but sometimes not quite. 

This is indeed the case when $R$ is a real Ginibre matrix, i.e. the entries of $R$ are independent $N(0,1)$
random variables. A result of Sankar, Spielman and Teng \cite{SST} states that 
\begin{equation}				\label{eq: SST}
\P \big\{ \|(D+R)^{-1}\| \ge t \sqrt{n} \big\} \le 2.35/t, \quad t > 0.
\end{equation}
In particular, $\|(D+R)^{-1}\| = O(\sqrt{n})$ with high probability. 
Note that this bound is independent of $D$. It is sharp for $D=0$, since 
$\|R^{-1}\| \gtrsim \sqrt{n}$ with high probability (\cite{Edelman}, see \cite{RV CRAS}).

For general non-Gaussian matrices $R$ a new phenomenon emerges: 
{\em invertibility of $D+R$ can deteriorate as $\|D\| \to \infty$}.

\medskip

Suppose the entries of $R$ are sub-gaussian\footnote{See \cite{V tutorial} for an introduction to sub-gaussian distributions. 
Briefly, a random variable $X$ is sub-gaussian if $p^{-1/2} (\E |X|^p)^{1/p} \le K < \infty$ for all $p \ge 1$;
the smallest $K$ can be called the sub-gaussian moment of $X$.} 
i.i.d. random variables with mean zero and variance one.
Then a result of Rudelson and Vershynin \cite{RV square} (as adapted by Pan and Zhou \cite{PZ})
states that as long as $\|D\| = O(\sqrt{n})$, one has 
$$
\P \big\{ \|(D+R)^{-1}\| \ge t \sqrt{n} \big\} \le C/t + c^n, \quad t > 0.
$$
Here $C>0$ and $c \in (0,1)$ depend only on a bound on the sub-gaussian moments of the entries
of $R$ and on $\|D\|/\sqrt{n}$.

Surprisingly, sensitivity to $\|D\|$ is not an artifact of the proof, but a genuine limitation.
Indeed, consider the example where each entry of $R$ equals $1$ and $-1$ with probability $1/4$ 
and $0$ with probability $1/2$. 
Let $D$ be the diagonal matrix with diagonal entries $(0, d, d, \ldots, d)$. 
Then one can show\footnote{This example is due to 
  M.~Rudelson (unpublished); a similar phenomenon was discovered independently by Tao and Vu \cite{TV smooth}.} 
that $\|(D+R)^{-1}\| \gtrsim d / \sqrt{n}$ with probability $1/2$.
In particular, $\|(D+R)^{-1}\| \gg \sqrt{n}$ as soon as $\|D\| = d \gg n$.

Note however that the typical value of $\|(D+R)^{-1}\|$ remains polynomial in $n$ as long as 
$\|D\|$ is polynomial in $n$. This result is due to Tao and Vu \cite{TV circ, TV STOC, TV smooth};
Nguyen \cite{Nguyen} proved a similar result for {\em symmetric} random matrices $R$.
\medskip

To summarize, as long as the deterministic part $D$ is not too large, $\|D\| = O(\sqrt{n})$,
Sankar-Spielman-Teng's invertibility bound \eqref{eq: SST} remains essentially valid 
for general random matrices $R$ (with i.i.d. subgaussian entries with zero mean and unit variance).
For very large deterministic parts ($\|D\| \gg n$), the bound can fail. 
It is not clear what happens in the regime $\sqrt{n} \ll \|D\| \lesssim n$.

Taking into account all these results, it would be interesting to describe
 ensembles of random matrices $R$ for which invertibility properties of $D+R$ are independent 
  of $D$.
In this note we show that if the entries of a symmetric matrix $R$ have {\em continuous} distributions,
then the typical value of $\|(D+R)^{-1}\|$ is polynomially bounded {\em independently of $D$}; 
in particular the bound does not deteriorate as $\|D\| \to \infty$.

\begin{theorem}\label{thm: main}
Let $A$ be an $n\times n$ symmetric random matrix in which the 
entries $\{A_{i,j}\}_{1\le i\le j\le n}$ are independent and have continuous distributions with densities bounded by $K$.  
Then for all $t>0$, 
\begin{equation}
\P\left\{ \|A^{-1}\| \ge n^2 t \right\}\le  8K/t.
\end{equation}
\end{theorem}

Since we do not assume that the entries have mean zero, this theorem can be applied
to matrices of type $A = D+R$, and it yields that $\|(D+R)^{-1}\| = O(n^2)$ with high probability. 
This bound holds for any deterministic symmetric matrix $D$, large and small. 
We conjecture that the bound can be improved to $O(\sqrt{n})$ 
as in Sankar-Spielman-Teng's result \eqref{eq: SST}.

\begin{remark}
  We do not place any upper bound assumptions in Theorem~\ref{thm: main}, 
  either on the deterministic part $D$ or the random part $R$. In particular, the entries of $R$ can be arbitrarily heavy-tailed. 
  The upper bound $K$ on the densities precludes the distributions concentrating near any value, 
  so effectively it is a lower bound on concentration.  
\end{remark}

\begin{remark}
  A result in the same spirit as Theorem~\ref{thm: main} was proved recently by Rudelson and Vershynin \cite{RV unitary} 
  for a different ensemble of random matrices $R$, namely for {\em random unitary matrices}. 
  If $R$ is uniformly distributed in $U(n)$ then 
  $$
  \P \big\{ \|(D+R)^{-1}\| \ge t n^C \big\} \le t^{-c}, \quad t > 0.
  $$
  As in Theorem~\ref{thm: main}, $D$ can be an arbitrary deterministic $n \times n$ matrix;
  $C, c>0$ denote absolute constants (independent of $D$). 
\end{remark}

\begin{remark}
  For the specific class where $D$ is a multiple of identity, sharper results are available than Theorem~\ref{thm: main}. 
  In particular, results by Erd\H{o}s, Schlein and Yau \cite{ESY Wegner} and 
  Vershynin \cite{V sym} yield an essentially optimal bound on the resolvent, $\|(D-zI)^{-1}\| = O(\sqrt{n})$. 
  Moreover, the latter estimate does not require that the entries of $D$ have continuous distributions; 
  see \cite{ESY Wegner, V sym} for details.
\end{remark}

\begin{remark}
  While Theorem~\ref{thm: main} is stated for symmetric matrices, 
  it holds as well for Hermitian matrices.
  The proof for the Hermitian case only requires an easy change to the proof of Lemma~\ref{boundentry} below. 
\end{remark}

\begin{remark}
  The proof of Theorem~\ref{thm: main} shows that one can relax the assumption of joint independence 
  of the entries. Is suffices to assume that the individual distribution of each entry $A_{ij}$, 
  conditioned on all other entries except $A_{ji}$, has density bounded by $K$.
\end{remark}

In the rest of the paper, we prove Theorem~\ref{thm: main}. 
The argument is very short and is based on computing the influence of each entry of $A$
on the corresponding entry of $A^{-1}$.

\section{Proof of Theorem~\ref{thm: main}}

Recall that the {\em weak $L_p$ norm} of a random variable $X$ is 
\begin{equation}
\|X\|_{p,\infty} :=\sup_{t>0}t \left(\P\{|X|>t\}\right)^{1/p}, \quad 0<p<\infty. 
\end{equation}

\begin{lemma}		\label{boundentry}
  Let $A$ be the random matrix defined in Theorem~\ref{thm: main}. 
  Then for all $1 \le i,j\le n$, 
  $$
  \|(A^{-1})_{i,j}\|_{1,\infty} \le 2K.
  $$
\end{lemma}

\begin{proof}
Let us determine how a single entry of the inverse, say $(A^{-1})_{i,j}$, depends on the corresponding entry 
of $A$, i.e. $A_{i,j}$. To this end, let us condition on all entries of $A$ except $A_{i,j}$, thus treating them as constants.
We could proceed by the cofactor expansion. But we find it easier to use Jacobi formula,
which is valid for an arbitrary square matrix $A = A(t)$ that depends on a parameter $t$:
$$
\frac{d}{dt}|A(t)|=\tr \big[ \adj(A(t)) \, \frac{d A(t)}{dt} \big].
$$
Here and later $|A|$ denotes the determinant and 
$\textnormal{adj}(A)$ denotes the adjugate matrix of $A$. 
Let $A_{(i,j)}$ be the submatrix obtained by removing the 
$i^{th}$ row and $j^{th}$ column of $A$, 
and let $A_{(i,j),(k,l)}$ be the submatrix obtained by removing rows $i$ and $k$ and columns $j$ and $l$ from $A$. 

Consider the off-diagonal case first, where $i\neq j$.
The Jacobi formula yields $\frac{d}{d A_{i,j}}|A_{(i,j)}|=(-1)^{i+j}|A_{(i,j),(j,i)}|$, so 
that 
\begin{equation}
|A_{(i,j)}|=(-1)^{i+j}|A_{(i,j),(j,i)}|A_{i,j}+a\label{linear}
\end{equation}
for some constant  $a$ (meaning that $a$ does not depend on $A_{i,j}$). 
Further, 
\begin{equation*}
\frac{d}{d A_{i,j}} |A|
= (-1)^{i+j}(|A_{(i,j)}|+|A_{(j,i)}|) = (-1)^{i+j} 2|A_{(i,j)}|
= 2|A_{(i,j),(j,i)}|A_{i,j}+(-1)^{i+j}2a.
\end{equation*}
Thus, for some constant $b$ one has 
\begin{equation}
|A|=|A_{(i,j),(j,i)}|A_{i,j}^2+(-1)^{i+j}2aA_{i,j}+b.		\label{quad}
\end{equation}

Equations~\eqref{linear} and~\eqref{quad} and Cramer's rule imply 
that for all $(i,j)$ there exist constants $p,q$ such that 
\begin{align*}
|(A^{-1})_{i,j}| 
= \left|\frac{|A_{(i,j)}|}{|A|}\right|
=\frac{|A_{i,j}+p|}{\big|(A_{i,j}+p)^2+q\big|}
= \Big| \frac{X}{X^2+q} \Big|, \qquad \text{where } X = A_{i,j}+p.
\end{align*}

First, assume that $q\geq 0$. Then $|(A^{-1})_{i,j}| \le 1/|X|$, and thus 
we have for all $t > 0$:
\begin{equation}					\label{off-diagonal}
\P\{|(A^{-1})_{i,j}|>t\} \le \P\{|X|<1/t\}\le 2K/t. 
\end{equation}

Next, assume $0>q=:-s$; then 
$$
|(A^{-1})_{i,j}| = \frac{1}{|X-s/X|}.
$$
Note that the function $f(x) := x-s/x$ satisfies $f'(x) = 1+s/x^2 > 1$ for all $x \ne 0$. 
Thus the set of points $\{ x \in \R: |f(x)| < \e\}$ has diameter at most $2\e$ for every $\e>0$. 
When $x=X$ is a random variable with density bounded by $K$, it follows that 
$\P \{ |f(X)| < \e \} \le 2 K \e$.
Using this for $\e=1/t$, we obtain
$$
\P\{|(A^{-1})_{i,j}| > t\} \le \P \{ |f(X)| < 1/t \} \le 2K/t.
$$
We have shown that in the off-diagonal case $i \ne j$, the estimate \eqref{off-diagonal} always holds. 

The diagonal case $i=j$ is similar. The Jacobi formula (or just expanding the determinant along $i$-th row)
shows that $|A|=|A_{(i,i)}|A_{i,i} + c$ for some constant $c$. Then a similar analysis yields
$\P\{|(A^{-1})_{i,j}| > t\} \le 2K/t$. This completes the proof.
\end{proof}

\begin{proof}[Proof of Theorem~\ref{thm: main}]
Although the weak $L_1$ norm is not equivalent to a norm, the following inequality holds for any 
finite sequence of random variables $X_i$: 
\begin{equation}				\label{eq: Hag}
\Big\| \big(\sum_i X_i^2 \Big)^{1/2} \Big\|_{1,\infty} \le 4 \sum_i \|X_i\|_{1,\infty}.
\end{equation}
This inequality is due to Hagelstein (see the proof of Theorem~2 in \cite{Hag}); it follows by 
a truncation argument and Chebyshev's inequality.
We use \eqref{eq: Hag} together with the estimates obtained in Lemma~\ref{boundentry}
to bound the Hilbert-Schmidt norm of $A$: 
$$
\big\| \|A^{-1}\|_{\mathrm{HS}} \big\|_{1,\infty}
=\Big\| \Big(\sum_{1\leq i,j\leq n}((A^{-1})_{i,j})^2 \Big)^{1/2}\Big\|_{1,\infty}
\leq 4\sum_{1\leq i,j\leq n}\|(A^{-1})_{i,j}\|_{1,\infty}
\le 8 K n^2.
$$

The definition of the weak $L_1$ norm then yields
$$
\sup_{t>0} t \P\{ \|A^{-1}\|_{\mathrm{HS}} > t\} \leq 8Kn^2. 
$$
Since $\|A^{-1}\| \le \|A^{-1}\|_{\mathrm{HS}}$, the proof of Theorem~\ref{thm: main} is complete.
\end{proof}

\bigskip

\noindent{\bf Acknowledgments.} 
We thank the referees whose suggestions helped to improve the presentation of this paper. 

B.~F. was partially supported by Joel A. Tropp under ONR awards N00014-08-1-0883 and N00014-11-1002 and a Sloan Research Fellowship.  R.~V. was partially supported by NSF grants 1001829, 1265782, and U. S. Air Force Grant FA9550-14-1-0009.



\end{document}